\documentclass[11pt]{article}
\usepackage{ae} 
\usepackage[T1]{fontenc}
\usepackage[ansinew]{inputenc}
\usepackage[leqno,namelimits]{amsmath}
\usepackage{amsthm}
\usepackage{amssymb}
\usepackage{graphicx}
\usepackage{fancyheadings}

\newtheorem{thm}{Theorem}[section]
\newtheorem{theorem}{Theorem}
\newtheorem*{lemma*}{Lemma}
\newtheorem{cor}[thm]{Corollary}
\newtheorem{lem}[thm]{Lemma}
\newtheorem{prop}[thm]{Proposition}
\newtheorem{defn}[thm]{Definition}
\newtheorem{fact}[thm]{Fact}
\newtheorem{notat}[thm]{Notation and terminology}

\numberwithin{equation}{section}

 \newcommand{\Real}{\mathbb{R}}
 \newcommand{\Complex}{\mathbb{C}}

 \newcommand{\set}[1]{\left\{#1\right\}}
 
 \newcommand{\norm}[1]{\left\Vert#1\right\Vert}

\author{E. Glakousakis and S. Mercourakis}

\title{On the existence of 1-separated sequences on the unit ball of a finite dimensional Banach space}
\begin{document}
\maketitle
\thispagestyle{fancy}
\lfoot{\vspace{1.2cm}\small{\textit{2010 Mathematics Subject Classification: Primary 46B20, 52A21; Secondary 52C99.}}}
\begin{abstract}
Given a finite dimensional Banach space $X$ and an Auerbach basis $\set{(x_{k},x_{k}^{*}): 1\leq k\leq n}$ of $X$, it is proved that there exist $n+1$ linear combinations $z_{1}, \ldots, z_{n+1}$ of $x_{1}, \ldots, x_{n}$ with coordinates $0, \pm 1$, such that $\norm{z_{k}}=1$, for $k=1, 2, \ldots, n+1$ and $\norm{z_{k}-z_{l}}>1$, for $1\leq k< l\leq n+1$.
\end{abstract}


\section*{Introduction}

\hspace*{0.5cm} In 1975 C. Kottman \cite{Kottman} proved that in each infinite dimensional Banach space $X$ there exists an infinite subset $\Delta$ of $S_{X}$ (the unit sphere of $X$) such that $\norm{x-y}>1$ for every pair $x, y \in \Delta$ with $x\neq y$. Actually, he proved that such a set always can be found in the set of linear combinations with coordinates $0,1,-1$ of any Auerbach system in $X$.

Let us denote by $c_{00}$ the linear space of real sequences which are eventually zero and by $e_{1}, \ldots, e_{k}, \ldots$ its usual basis. We also set $\mathcal{U}=c_{00}\bigcap\set{0,\pm 1}^{\mathbb{N}}$. The core of Kottman's argument is the following lemma.
\begin{lemma*}\emph{(Kottman)}\label{lem1.1}
There is no subset $A$ of $\mathcal{U}$ such that
\begin{itemize}
  \item [$(a)$] $e_{k}\in A$, for $k=1, 2, \ldots$
  \item [$(b)$] $A$ is symmetric, that is, if $x\in A$, then $-x\in A$ and
  \item [$(c)$] if $B$ is any infinite subset of $A$, then there exist $x, y\in B$ with $x\neq y$ such that $x-y\in A$.
\end{itemize}
\end{lemma*}
It is immediate that the above Lemma can be stated as follows.
\\ \\
\textit{If $A$ is any subset of $\mathcal{U}$ satisfying conditions (a) and (b), then there exists an infinite $B\subseteq A$ such that if $x, y\in B$ and $x\neq y$, then $x-y\notin A$.} Thus $A$ contains an infinite subset $B$ which is difference free with respect to $A$.
\\ \\
The aim of this note is to investigate the validity of analogous results in finite dimensional spaces. The note is divided in two sections. The main results of the first section are the following.
\begin{theorem}\label{th1.2}
For every $n\in\mathbb{N}$, there is no subset of $A$ of the cube $C_{n}=\set{0,\pm 1}^{n}$ such that
\begin{itemize}
  \item [$(a)$] $e_{k}\in A$, for $1\leq k\leq n$,
  \item [$(b)$] $A$ is symmetric, that is, if $x\in A$, then $-x\in A$ and
  \item [$(c)$] for every subset $B$ of $A$ with $|B|=n+1$  there exist distinct $x, y\in B$  such that $x-y\in A$.
\end{itemize}
\end{theorem}
So if $n\in\mathbb{N}$ and $A$ is any subset of $C_{n}$ satisfying conditions $(a)$ and $(b)$, then there exists $B\subseteq A$ with $|B|=n+1$
which is difference-free with respect to $A$.\\ \\
Subsequently using Theorem \ref{th1.2} we can prove in a similar way as Kottman proved his main result, the following separation result for finite dimensional Banach spaces.
\begin{theorem}\label{th1.3}
Let $X$ be a finite dimensional Banach space with $\dim X=n$. Then for any Auerbach basis $\set{(x_{k},x_{k}^{*}): 1\leq k\leq n}$ of $X$, there exist $n+1$ linear combinations $z_{1}, \ldots, z_{n+1}$ of the vectors $x_{1}, \ldots, x_{n}$ with coordinates $0, \pm 1$ such that $\norm{z_{k}}=1$, for $1\leq k\leq n+1$ and $\norm{z_{k}-z_{l}}>1$, for $1\leq k<l\leq n+1$.
\end{theorem}
In the final part of section 1, we investigate the relationship between Kottman's Lemma and its finite version, which is Theorem \ref{th1.2}. (See Prop. \ref{prop2.9} and the proof of Kottman's Lemma by Prop. \ref{prop2.5}.) \\ \\
In the second section  we are concerned with some other versions of Theorem \ref{th1.2}, that is, its dual version concerning sum-free sets (see Theorem \ref{th1.4} below) and a complex version of it (Theorem \ref{th3.9}). These results have consequences similar to Theorem \ref{th1.3} (see Theorems \ref{th3.3} and \ref{th3.10}).
\begin{theorem}\label{th1.4}
For any $n\in\mathbb{N}$ and each $A\subseteq C_{n}$ satisfying conditions:
\begin{itemize}
  \item [$(a)$] $e_{i}\in A$, for $i=1, \ldots, n$ and
  \item [$(b)$] $A$ is symmetric.
\end{itemize}
There exist a subset $B\subseteq A$ such that $|B|=n$ and $x+y\notin A$ for every $x, y\in B$. (That is, $B$ is sum-free with respect to $A$.)
\end{theorem}
Regarding Theorem \ref{th1.3} we mention here that in \cite{Arias} it was shown that, in every finite dimensional Banach space ``most'' of the pairs of points in the unit ball (endowed with normalized Lebesgue measure) are separated by more than $\sqrt{2}(1-\varepsilon)$. In particular we may find an exponentially  large number of points on the unit sphere any two of which are separated so. This fact was proved by Bourgain using another method. We also remark that Elton and Odell \cite{Elton}, improving the result of Kottman (Theorem \ref{th2.2} below) have proved that the unit sphere of every infinite dimensional Banach space contains a $(1+\varepsilon)$-separated sequence.
\section{A finite version of Kottman's Lemma}

\hspace*{0.5cm} We start by proving Kottman's separation theorem using his Lemma. We first remind the following definition.
\begin{defn}
Let $X$ be a Banach space. A system of pairs $\set{(x_{\gamma},x_{\gamma}^{*}): \gamma\in\Gamma}\subseteq X\times X^{*}$ is said to be an Auerbach system if
\begin{itemize}
  \item [$(a)$] it is biorthogonal, i.e. $x_{\alpha}^{*}(x_{\beta})=\delta_{\alpha\beta}$, for $\alpha, \beta\in\Gamma$ and
  \item [$(b)$] $\norm{x_{\gamma}}=\norm{x_{\gamma}^{*}}=1$, for $\gamma\in\Gamma$.
\end{itemize}
\end{defn}
It is a classical fact that every finite dimensional Banach space $X$ has an Auerbach basis, that is, an Auerbach system $\set{(x_{k},x_{k}^{*}): 1\leq k\leq n}$ where $n=\dim X$. By a result of Day, in any infinite dimensional Banach space $X$ we can always find an Auerbach system  $\set{(x_{k},x_{k}^{*}): k\in\mathbb{N}}$ \cite{Hajek}.
\begin{thm}\emph{(Kottman)}\label{th2.2}
Let $X$ be a Banach space with $\dim X=\infty$ and $\set{(x_{k},x_{k}^{*}): k\in\mathbb{N}}$ an Auerbach system in $X$. Let also $E=\set{x=\sum_{k=1}^{n}a_{k}x_{k}: \right.$ $\left. n\in \mathbb{N}, \norm x=1 \,\,\,\mathrm{and}\,\,\, a_{k}\in\set{0,\pm 1}}$. Then there is an infinite subset $\Delta$ of $E$ such that $\norm{x-y}>1$ for every $x, y\in\Delta$ with $x\neq y$.
\end{thm}
\begin{proof}
Let $Y$ be the linear span of the set $\set{x_{k}: k\in\mathbb{N}}$, we define the operator $T: Y\to (c_{00},\norm{\cdot}_{\infty})$ by $T(x)=(x_{k}^{*}(x))$, for $x\in Y$. It is easy to check that $T$ is a linear one to one operator such that $T(E)\subseteq \mathcal{U}$ and $\norm{T}=1$. So we have that
\[
\norm{T(x)-T(y)}_{\infty}\leq\norm{x-y}, \,\,\mathrm{for}\,\,\, x, y \in Y.
\]
Through the above inequality we observe that if $x,y \in E$ with $x\neq y$  and $\norm{x-y}\leq1$, then $x-y\in E$. We put $A=T(E)$ then $A$ satisfies conditions $(a)$ and $(b)$ of Kottman's Lemma, therefore, by its restatement, there exists an infinite subset $B$ of $A$ such that, if $\Delta=T^{-1}(B)$, then we have: $x, y\in\Delta$ and $x\neq y$, then $T(x)-T(y)\notin A$, which clearly implies that $\norm{x-y}>1$.
\end{proof}
Since the proof of Theorem \ref{th1.3} follows the lines of the proof of Theorem \ref{th2.2}, we give a brief description of it.
\\ \\
\textit{Proof of Theorem \ref{th1.3}.}
Let $\set{(x_{k},x_{k}^{*}): 1\leq k\leq N}$ be any Auerbach basis of $X$. We set $E_{N}=\set{x=\sum_{k=1}^{N}a_{k}x_{k}:  \norm x=1 \,\,\,\mathrm{and}\,\,\, a_{k}\in\set{0,\pm 1}}$ and consider the invertible linear operator $T: X\to l_{\infty}^{N}$, defined by $T(x)=(x_{k}^{*}(x))$, for $x\in X$. We observe that $T(E_{N})\subseteq C_{N}$ and simply apply Theorem \ref{th1.2} to get the conclusion.
\\

In order to prove Theorem \ref{th1.2}, we are going to use the following.
\begin{notat}\label{not1.3}
\textbf{(1)} Let $l\geq2$, for $N\geq 2$ we write $N\to l$ if there is no subset $A$ of $C_{N}$ such that:
\begin{itemize}
  \item [$(a)$] $e_{k}\in A$, for $k=1, \ldots, N$
  \item [$(b)$] $A$ is symmetric and
  \item [$(c)$] for every subset $B$ of $A$ with $|B|=l$  there exist $x\neq y\in B$  such that $x-y\in A$.
\end{itemize}
\textbf{(2)} The Kottman function $K(l)$, for $l\geq2$ denotes the minimal $N\in\mathbb{N}$ such that $N\to l$. It is easy to see that $1=K(1)\leq K(l)\leq K(l+1)$, for $l\in\mathbb{N}$.\\
\textbf{(3)} Let $A\subseteq C_{m}$, for $n\leq m$ we denote by $pr_{n}A$ the projection of $A$ at the first n-coordinates. For $x=(x_{1}, \ldots, x_{n})\in pr_{n}A$ any vector of the form $(x,\xi)=(x_{1}, \ldots, x_{n}, \xi)\in A$ will be called an extension of $x$ in $A$. The notation $pr_{n}A$ for any subset $A$ of $\mathcal{U}$ has a similar meaning.
\end{notat}
\begin{fact}\label{fact1.4}
Let $A\subseteq C_{N+1}$ and $x\neq y\in pr_{N}A$ such that $x-y \notin pr_{N}A$. Then for any extensions $(x, \xi), (y, \zeta)$ of $x$ and $y$ in $A$, $(x, \xi)-(y,\zeta)\notin A$.
\end{fact}
\begin{proof}
Assuming the contrary we would have that $(x-y,\xi-\zeta)\in A$ and consequently that $x-y\in pr_{N}A$, a contradiction.
\end{proof}
\begin{prop}\label{prop2.5}
Let $N, l\in\mathbb{N}$ and suppose that $N\to l$. Let also $A$ be a subset of $C_{N+1}$ satisfying (a) and (b) and $B=\set{x_{1}, \ldots, x_{l}}$ a subset of $pr_{N}A$ such that $x_{i}-x_{j}\notin pr_{N}A$, for $1\leq i<j\leq l$. Then there exist extensions $(x_{1}, \xi_{1}), \ldots, (x_{l},\xi_{l})$ of the elements of $B$ in $A$ and $z\in A$ such that $x-y\notin A$, for every $x\neq y \in B'=\set{(x_{1}, \xi_{1}), \ldots, (x_{l}, \xi_{l}),z}$.
\end{prop}
\begin{proof}
Suppose that for any extensions of the elements of $B$ in $A$ and for any $z\in A$ for the corresponding set $B'$ we have that there exist $x\neq y\in B'$ such that $x-y\in A$ (1).
\\ \\
\underline{\textit{Claim 1.}}
\textit{There exists $x\in B$ such that $(x,1)\in A$.}
\\ \\
\textit{Proof of Claim 1.}
We assume that $(x,1)\notin A$, for every $x\in B$. For $1\leq i\leq l$ we consider an extension $(x_{i},\xi_{i})$ of $x_{i}$ in $A$ and put $B_{1}=\set{(x_{1},\xi_{1}), \ldots, (x_{l},\xi_{l}), (0,1)}$ (of course $\xi_{i}$ must belong to $\set{0,-1}$ for $1\leq i\leq l$). By (1) we take that there exist $x\neq y\in B_{1}$ such that $x-y \in A$ and from Fact \ref{fact1.4} that means that there exists $1\leq i_{0}\leq l$ such that $(x_{i_{0}}, \xi_{i_{0}})-(0,1)\in A$. Since $\xi_{i}\in\set{0,-1}$, for $1\leq i\leq l$ it is immediate that $\xi_{i_{0}}$ must be 0. Now put
\begin{eqnarray*}
\set{x_{1}, \ldots, x_{k}} & = & \set{x\in B: (x,0), (x,-1)\in A} \\
\set{x_{k+1}, \ldots, x_{m}} & = & \set{x\in B: (x,0)\in A, (x,-1)\notin A}\\
\set{x_{m+1}, \ldots, x_{l}} & = & \set{x\in B: (x,0)\notin A, (x,-1)\in A}
\end{eqnarray*}
and
\[
\zeta_{i}=-1,\,\,\,  \textrm{for}\,\,\, 1\leq i\leq k \,\,\,\textrm{or}\,\,\, m+1\leq i\leq l\,\,\, \textrm{and}\,\,\, \zeta_{i}=0, \,\,\,\textrm{for}\,\,\, k+1\leq i\leq m.
\]
Then $(x_{i},\zeta_{i})\in A$, for $1\leq i\leq l$. Considering the set $B_{2}=\set{(x_{1},\zeta_{1}), \ldots, (x_{l}, \zeta_{l}), (0,1)}$ we conclude, as before, that there exists $1\leq i\leq l$ such that $(x_{i}, \zeta_{i}-1)\in A$, which is a contradiction due to the choice of $\zeta_{i}$.
\\ \\
\underline{\textit{Claim 2}}
\textit{There exists $x\in B$ such that $(x,\xi)\in A$ for $\xi\in\set{0,\pm 1}$}.
\\ \\
\textit{Proof of Claim 2}.
From Claim 1 we have that $\set{x\in B: (x,1)\in A}\neq\emptyset$. Then there exist $1\leq k< m\leq l$ such that
\begin{eqnarray}\label{eq2}
\nonumber \set{x_{1}, \ldots, x_{k}} & = & \set{x\in B: (x,1)\in A} \\
\set{x_{k+1}, \ldots, x_{m}} & = & \set{x\in B: (x,1)\notin A, (x,-1)\in A}\\
\nonumber \set{x_{m+1}, \ldots, x_{l}} & = & \set{x\in B: (x,1)\notin A, (x,-1)\notin A, (x,0)\in A}.
\end{eqnarray}
Now for $1\leq i \leq l$ we put $\xi_{i}=1$, for $1\leq i\leq k$, $\xi_{i}=-1$, for $k+1\leq i\leq m$ and $\xi_{i}=0$, for $m+1\leq i\leq l$. Further we consider the subset of $A$, $B_{3}=\set{(x_{1},\xi_{1}), \ldots, (x_{l},\xi_{l}),(0,1)}$. Again by (1) we have that there exists $1\leq i\leq l$ such that $(x_{i},\xi_{i}-1)\in A$. From \eqref{eq2} and the choice of $\xi_{i}$ it follows that
\begin{equation}\label{eq3}
i\leq k.
\end{equation}
From the last we take that there exists also $1\leq n\leq k$ such that
\begin{eqnarray}\label{eq4}
\nonumber\set{x_{1}, \ldots, x_{n}} & = & \set{x\in B: (x,1), (x,0)\in A} \\
\set{x_{n+1}, \ldots, x_{k}} & = & \set{x\in B: (x,1)\in A, (x,0)\notin A} \\
\nonumber \set{x_{k+1}, \ldots, x_{m}} & = & \set{x\in B: (x,0)\in A, (x,-1)\notin A}\\
\nonumber\set{x_{m+1}, \ldots, x_{l}} & = & \set{x\in B: (x,0)\notin A, (x,-1)\in A}.
\end{eqnarray}
We put now $\zeta_{i}=0$, for $1\leq i\leq n$, $\zeta_{i}=1$, for $n+1\leq i\leq k$, $\zeta_{i}=0$, for $k+1\leq i\leq m$ and $\zeta_{i}=-1$, for $m+1\leq i\leq l$. Then the set $B_{4}=\set{(x_{1},\zeta_{1}), \ldots, (x_{l},\zeta_{l}),(0,1)}$ is a subset of $A$ and from (1) follows that there exists $1\leq i_{0}\leq l$ such that $(x_{i_{0}},\zeta_{i_{0}}-1)\in A$. Finally, \eqref{eq4} and the choice of $\zeta_{i}$ yield that $i_{0}\leq n$ and so $(x_{i_{0}},\xi)\in A$, for every $\xi\in\set{0,\pm 1}$. \\ \\
Having proved Claim 2, we may assume that $(x_{1},\xi)\in A$, for $\xi\in\set{0,\pm 1}$. We consider now extensions of the rest elements of $B$ in $A$, $(x_{2},\xi_{2}), \ldots, (x_{l},\xi_{l})$ and the set $B'=\set{(x_{1},1), (x_{1},-1), (x_{2},\xi_{2}), \ldots, (x_{l},\xi_{l})}$. Fact 2.4 and (1) yield now $(0,2)=(x_{1},1)-(x_{1},-1)\in A$, a contradiction.
\end{proof}
\begin{cor}\label{cor2.6}
Let $N, l\in\mathbb{N}$.
\begin{itemize}
\item [$(a)$] If $N\to l$, then $N+1\to l+1$.
\item [$(b)$] Suppose that the number $K(l)$ exists, then $N\to l$, for every $N\geq K(l)$.
\item [$(c)$] $K(l)$ is well defined, moreover $K(l)\leq l-1$, for $l\geq2$.
\end{itemize}
\end{cor}
\begin{proof}
The proof of $(a)$ is immediate from Proposition \ref{prop2.5}. Also it is immediate that $(a)$ yields $(b)$. For $(c)$ just observe that $K(2)=1$ and proceed inductively using $(a)$.
\end{proof}
In Corollary \ref{cor2.6} we proved that the function $K(l)$ is well defined and has the number $l-1$ as an upper bound. Actually the number $l-1$ is the exact value of $K(l)$. To prove that we will need the following Lemma.
\begin{lem}\label{lem2.7}
For $x=(k,m)\in\set{1, \ldots, n}^{2}\setminus\set{(k,k): 1\leq k\leq n}$ we denote with suppx the set $\set{k,m}$. Let $B$ be a subset of $\set{1, \ldots, n}^{2}\setminus\set{(k,k): 1\leq k\leq n}$ with the following property:
\begin{itemize}
\item [$(*)$] for $x\neq y\in B$ either
\begin{equation}\label{eq2.4}
suppx \cap suppy=\emptyset, \,\,\,\, or
\end{equation}
\begin{equation}\label{eq2.5}
suppx\cap suppy\neq\emptyset
\end{equation}
and if $k\in suppx\cap suppy$, then $k$ is the first coordinate of either $x$ or $y$ and the second of the other.
\end{itemize}
Then we have
\begin{itemize}
\item [$(a)$] $|B|\leq n$ and
\item [$(b)$] if $|B|=n$, then for every $k\in\set{1, \ldots, n}$ there exist $x\neq y\in B$ such that $k\in suppx\cap suppy$.
\end{itemize}
\end{lem}
\begin{proof}
For $k\in\set{1, \ldots, n}$ we consider the vertical lines $L_{k}^{v}=\set{(k,m):\right.$ $\left. m\in\set{1, \ldots, n}}$ and the parallel lines $L_{k}^{p}=\set{(m,k): m\in\set{1, \ldots, n}}$. Let $B$ be a subset of $\set{1, \ldots, n}^{2}\setminus\set{(k,k): k\in\set{1, \ldots, n}}$ with the property $(*)$.
\begin{itemize}
  \item [$(a)$] It is direct from \eqref{eq2.5} that $B$ has at most one point in any vertical or parallel line, so $|B|\leq n$.

      Now we observe that $B$ intersects the same number of parallel and vertical lines. Indeed if $B$ intersects $k$ vertical and $m$ parallel lines with $k<m$, then by \eqref{eq2.5} one of the vertical lines contains at least two points of $B$, a contradiction. A set of maximum cardinality with the property $(*)$ is the following $\set{(1,2), \ldots, (n-1,n), (n,1)}$.
  \item [$(b)$] Let $B$ be a subset of  $\set{1, \ldots, n}^{2}\setminus\set{(k,k): k\in \set{1, \ldots, n}}$ of maximum cardinality $|B|=n$, with the property $(*)$. It follows that $B$ has at most one point in every vertical line and exactly one point in every parallel line. The last yields assertion $(b)$.
\end{itemize}
\end{proof}
It should be clear now, given the notation we have used, that Theorem \ref{th1.2} of the introduction is an easy consequence of the following.
\begin{thm}\label{th2.8}
$K(n+1)=n$, for $n\in\mathbb{N}$.
\end{thm}
\begin{proof}
Set $l=n+1$. By Corollary \ref{cor2.6} we have that $K(l)\leq l-1$, for every $l\geq2$. So it is enough to show that $l-2<K(l)$, for every $l>2$ or equivalently that for every $l\geq2$ there exists a subset $A$ of $C_{l-2}$ such that
\begin{enumerate}
  \item [(1)] $e_{i}\in A$, for $1\leq i\leq l-2$
  \item [(2)] $A$ is symmetric
  \item [(3)] there is no subset $B$ of $A$ with $|B|=l$ and $x-y\notin A$, for every $x\neq y \in B$.
\end{enumerate}
Suppose the contrary, so there exists $l>2$ such that for every subset $A$ of $C_{l-2}$ which satisfies conditions (1) and (2) there is a subset $B$ of $A$ such that $|B|=l$ and $x-y\notin A$ for every $x\neq y\in B$. Let $l>2$ be as above. We consider the sets $A_{1}=\set{\pm e_{i}: 1\leq i\leq l-2}$ and $A_{2}=\set{e_{i}-e_{j}: 1\leq i\neq j\leq l-2}$. It is direct to see that $A_{1}\cap A_{2}=\emptyset$ and that the set $A=A_{1}\cup A_{2}$ is a subset of $C_{l-2}$ which satisfies conditions (1) and (2). For $x\in A$ and $i\in suppx$ we will say that $e_{i}$ is of positive (resp. negative) sign in $x$ if $x$ is of the form $e_{i}$ (resp. $-e_{i}$)  or $e_{i}-e_{j}$ (resp. $e_{j}-e_{i}$) for some $j\in\set{1, \ldots, l-2}$. Let now $B$ be a subset of $A$ such that $|B|=l$ and $x-y\notin A$ for every $x\neq y\in B$. We observe that $|B\cap A_{1}|\leq2$. Indeed assuming that there exist distinct $x, y, z \in B\cap A_{1}$ then two of them, say $x$ and $y$, have the same sign, so we have that either $x=e_{i}$ and $y=e_{j}$ or $x=-e_{i}$ and $y=-e_{j}$. Any possible result we take is $x-y\in A$, a contradiction. The above observation yields that $|B\cap A_{2}|\geq l-2$ \eqref{eq4}.
\\ \\
\underline{\textit{Claim.}} \textit{Let $x\neq y\in A_{2}$, then $x-y\notin A$ if and only if
\\ \\
either (a) $suppx\cap suppy=\emptyset$
\\ \\
or (b) $suppx\cap suppy\neq \emptyset$ and if $i\in suppx\cap suppy$, then $e_{i}$ is of positive sign in either $x$ or $y$ and of negative sign in the other.
\\ \\
Proof of the Claim.}
For the direction $(\Rightarrow)$ it is enough to show that if assertion $(a)$ does not hold, then assertion $(b)$ holds. Let $x\neq y\in A_{2}$ such that $suppx\cap suppy\neq\emptyset$. Let $i\in suppx\cap suppy$. If $e_{i}$ is of the same sign in $x$ and $y$, let us assume of positive sign, we have that $x=e_{i}-e_{j}$, $y=e_{i}-e_{k}$, for some $j\neq k$ and consequently that $x-y=e_{k}-e_{j}\in A$, a contradiction. Should we have that $suppx\cap suppy=\emptyset$ for some $x, y\in A_{2}$ it is direct to show that $x-y\notin A$. So for the converse implication it is enough to show that if $x\neq y\in A_{2}$ for which assertion $(b)$ holds, then $x-y\notin A$. The last is just a matter of calculations.\\
We consider the mapping $f: A_{2}\to \set{1, \ldots, l-2}^{2}$,
\[
f(e_{i}-e_{j})=(i,j).
\]
It is easy to check that $f$ is an one-to-one mapping and that by the above Claim the image $f(B\cap A_{2})$ of $B\cap A_{2}$ satisfies $(*)$ of Lemma \ref{lem2.7}. So by Lemma \ref{lem2.7}, $|B\cap A_{2}|\leq l-2$. Now by \eqref{eq4} we take $|B\cap A_{2}|=l-2$ and $|B\cap A_{1}|=2$. Let $i\in\set{1, \ldots, l-2}$ such that $e_{i}\in B\cap A_{1}$. Again by Lemma \ref{lem2.7} we take that there exists $1\leq i\neq j\leq l-2$ such that $e_{i}-e_{j}\in B\cap A_{2}$. Then $e_{i}-(e_{i}-e_{j})=e_{j}$ must not belong to $A$, a contradiction.
\end{proof}
The following two results concern the connection between Kottman's Lemma and the discussed above finite version of it. First in Proposition \ref{prop2.9} we prove using Kottman's Lemma, that the function $K(l)$ is well defined, then we give an alternative proof of Kottman's Lemma using Proposition \ref{prop2.5}. These results show that our approach based on Proposition \ref{prop2.5}, is somewhat stronger than this of Kottman.
\begin{prop}\label{prop2.9}
The function $K(l)$ is well defined.
\end{prop}
\begin{proof}
Assume that there is an $l\geq2$ such that $N\nrightarrow l$ for every $N\geq l$. So for $N\geq l$ we can choose $A_{N}\subseteq C_{N}$ which satisfies the conditions $(a)$, $(b)$ and $(c)$ of \ref{not1.3}. Now for $N\geq l$ and $l\leq m\leq N$ put $A_{N}^{m}=\set{x\in A_{N}: \max suppx \leq m}$. We can consider the set $A_{N}^{m}$ as a subset of $C_{m}$ which satisfies the conditions $(a)$, $(b)$ and $(c)$ of \ref{not1.3} for any $l\leq m\leq N$.
Since for every $N\geq l$, $A_{N}^{l}$ is a subset of $C_{l}$ and the powerset of $C_{l}$ is finite, we can choose an infinite subset $M_{0}$ of $\mathbb{N}$ and a subset $\Gamma_{0}$ of $C_{l}$ such that
\[
A_{N}^{l}=\Gamma_{0}, \,\,\, \textrm{for}\,\,\, \textrm{every}\,\,\, N\in M_{0}.
\]
By the same argument we can inductively choose a decreasing sequence $(M_{j})_{j\geq0}$ of infinite subsets of $\mathbb{N}$ and an increasing sequence of sets $(\Gamma_{j})_{j\geq0}$ such that
\begin{enumerate}
  \item [(1)] $\Gamma_{j}\subseteq C_{l+j}$, for every $j\geq0$
  \item [(2)] $A_{N}^{l+j}=\Gamma_{j}$, for every $N\in M_{j}$.
\end{enumerate}
For $j\geq0$, we can consider the set $\Gamma_{j}$ as a subset of $c_{00}$ by adding zero coordinates after the $(l+j)$-th coordinate of $x$, for every $x\in\Gamma_{j}$. Under this consideration we can easily check that
\begin{enumerate}
\item [(3)] $e_{i}\in\Gamma_{j}$, for $1\leq i\leq l+j$ (we now refer to $e_{i}\in c_{00}$)
\item [(4)] $\Gamma_{j}$ satisfies the conditions $(b)$ and $(c)$ of \ref{not1.3} for $j\geq0$.
\end{enumerate}
Now putting $\Gamma=\bigcup_{j=0}^{\infty}\Gamma_{j}$ we observe that
\begin{enumerate}
\item [(5)] $e_{i}\in\Gamma$, for $i\in\mathbb{N}$
\item [(6)] $\Gamma$ is symmetric.
\end{enumerate}
Let $B$  be an infinite subset of $\Gamma$ and $B'$ any subset of $B$  with $|B'|=l$. By the construction of $\Gamma$ there must be $j\geq0$ such that  $\max suppx\leq l+j$, for every $x\in B'$. So $B'\subseteq\Gamma_{j}$. Since $\Gamma_{j}$ satisfies $(c)$ of \ref{not1.3} there must be $x\neq y\in B'\subseteq B$ such that $x-y\in\Gamma_{j}\subseteq\Gamma$, which is a contradiction by (5), (6) and Kottman's Lemma.
\end{proof}
\begin{prop}\label{prop2.10}
Proposition \ref{prop2.5} implies Kottman's Lemma.
\end{prop}
\begin{proof}
Let $A$ be a subset of $\mathcal{U}$ such that satisfies conditions $(a)$ and $(b)$ of Kottman's Lemma. We put $A_{n}=pr_{n}A$, for $n\in\mathbb{N}$. It is direct that if $n\leq m$, then $pr_{n}A_{m}=A_{n}$. Notice that $A_{n}$ is a symmetric subset of $C_{n}$ such that $e_{k}\in A_{n}$, for every $1\leq k\leq n$. We next define a sequence of sets $(B_{n})$ such that
\begin{equation}\label{eq2.6}
B_{n}\subseteq A_{n}, \,\,\textrm{for}\,\,\, \textrm{every}\,\,\, n\in \mathbb{N}
\end{equation}
\begin{equation}\label{eq2.7}
|B_{n}|=n+1\,\,\, \textrm{and}\,\,\, x-y\notin A_{n}, \,\,\textrm{for} \,\,\, \textrm{every}\,\, x\neq y\in B_{n}
\end{equation}
\begin{equation}\label{eq2.8}
if\,\, n\leq m, B_{n}\subseteq pr_{n}B_{m}
\end{equation}
Suppose that $B_{n}$ has already been defined. Since $pr_{n}A_{n+1}=A_{n}$ property \eqref{eq2.7} and Proposition \ref{prop2.5} yield that there exists $B_{n+1}\subseteq A_{n+1}$ comprised from extensions of elements of $B_{n}$ in $A_{n+1}$ and an element of $A_{n+1}$ such that $x-y\notin A_{n+1}$ for every $x\neq y\in B_{n+1}$. Put $B=\set{x\in A: pr_{n}x\in B_{n}\,\,\right.$ $\left.\textrm{eventually\,\,\, for\,\,\, every}\,\,\, n}$, then $B$ is an infinite subset of $A$. Now it suffices to show that $x-y\notin A$, for every $x\neq y\in B$. Indeed assume that there exist $x\neq y\in B$ such that $x-y\in A$. Choose $n\in\mathbb{N}$  such that $pr_{n}x, pr_{n}y\in B_{n}$ and $pr_{n}x\neq pr_{n}y$. Since $x-y\in A$ we have that $pr_{n}(x-y)=pr_{n}x-pr_{n}y \in A_{n}$, a contradiction.
\end{proof}

\section{The sum-free and the complex version of Kottman's Lemma}

\hspace*{0.5cm} In this section we state and prove results in the spirit of Theorems \ref{th1.2} and \ref{th1.3} of the introduction, the sum-free sets case and the complex case. \\ 

We remind the definition of a sum-free set.
\begin{defn}
Let $(G,+)$ be an abelian group and $A$ a non empty subset of $G$. A non empty subset $B$ of $A$ will be called sum-free with respect to $A$ if $x+y\notin A$, for every $x, y \in B$ with $x\neq y$.
\end{defn}
It should be clear that Theorem \ref{th3.2} below is a stronger version of Theorem \ref{th1.4} of the introduction. For the needs of this theorem we will use the following notation.\\ \\
\textit{For $N, l\in \mathbb{N}$, we will write $N\xrightarrow[s]{}l$, if there is no subset $A$ of $C_{N}$ such that:
\begin{itemize}
\item [$(a)$] $e_{k}\in A$, for $1\leq k\leq N$
\item [$(b)$] $A$ is symmetric
\item [$(c)$] for every subset $B$ of $A$ with $|B|=l$ there exist $x, y\in B$ such that $x+y\in A$.
\end{itemize}
$S(l)$ denotes the minimal $N\in\mathbb{N}$ such that $N\xrightarrow[s]{} l$. It is clear that if $N\geq S(l)$ then every subset $A$ of $C_{N}$ satisfying (a) and (b) contains a subset $B$ with $|B|=l$, which is sum-free with respect to $A$. We notice that $1=S(1)\leq S(l)\leq S(l+1)$, for $l\in \mathbb{N}$.
}
\begin{thm}\label{th3.2}
(1) Let $l\in\mathbb{N}$. If $N\xrightarrow[s]{}l$, then $N+1\xrightarrow[s]{}l+1$.\\
(2) $S(l)=l$, for every $l\geq 1$.
\end{thm}
\begin{proof}
(1) Assume that $N\xrightarrow[s]{}l$ and $N+1\nrightarrow_{s} l+1$. Then there exists a subset $A$ of $C_{N+1}$ satisfying conditions $(a)$, $(b)$ and $(c)$. Observe that $pr_{N}A$ is a subset of $C_{N}$ satisfying conditions $(a)$ and $(b)$. Since $N\xrightarrow[s]{}l$ there exists a subset $B$ of $pr_{N}A$ such that $|B|=l$ and $x+y\notin pr_{N}A$, for every $x, y\in B$.\\ \\
\underline{\textit{Claim}}
\textit{There exists $x\in B$ such that $(x,1)\in A$.}
\\ \\
\textit{Proof of Claim.}
The proof of this Claim is similar to the proof of Claim 1 of Proposition \ref{prop2.5} so it is omitted.
\\ \\
It follows from the Claim that there exists $1\leq k\leq l$ such that
\begin{itemize}
\item [$(I)$] $\set{x_{1}, \ldots, x_{k}}=\set{x\in B: (x,1)\in A}$.
\end{itemize}
Let $m>k$ such that
\begin{itemize}
\item [$(II)$] $\set{x_{k+1}, \ldots, x_{m}}=\set{x\in B: (x,1)\notin A \,\,\, and\,\,\,(x,0)\in A}$\\
    $\set{x_{m+1}, \ldots, x_{l}}=\set{x\in B: (x,1), (x,0)\notin A \,\,\,and\,\,\,(x,-1)\in A}.$
\end{itemize}
Notice that the sets defined in (II) may be empty. Put now $\xi_{i}=1$, for $1\leq i\leq k$, $\xi_{i}=0$, for $k+1\leq i\leq m$ and $\xi_{i}=-1$, for $m+1\leq i\leq l$. Then $(x_{i},\xi_{i})\in A$, for $1\leq i\leq l$. So the set $B'=\set{(x_{1},\xi_{1}), \ldots, (x_{l},\xi_{l}), (0,1)}$ is a subset of $A$ with $|B'|=l$ and consequently there must be $x,y\in B'$ such that $x+y\in A$. As in Fact 2.4, it can be proved that $(x_{i}+x_{j},\xi_{i}+\xi_{j})\notin A$ for every $1\leq i, j\leq l$ which yields that there exists $1\leq i\leq l$ such that $(x_{i},\xi_{i}+1)\in A$, a contradiction due to (I), (II) and the choice of $\xi_{i}$.
\\ \\
(2) From (1) we take that $S(l)\leq l$ for $l\geq2$, also notice that $S(2)=2$. So it suffices to show that $S(l)>l-1$ for $l>2$. Assuming the contrary there exists $l>2$ such that for every subset $A$ of $C_{l-1}$ satisfying conditions $(a)$ and $(b)$ there exists a subset $B$ of $A$, with $|B|=l$, such that $x+y\notin A$ for every $x, y\in B$. Put $A=\set{\pm e_{i}: 1\leq i\leq l-1}\cup\set{0}$. Now, as we may, we choose a subset $B$ of $A$ such that $|B|=l$ and $x+y\notin A$ for every $x, y\in B$. Observe that $0\notin B$ so $B$ is a subset of $\set{\pm e_{i}: 1\leq i\leq l-1}$. We claim that there exists at least one $1\leq i\leq l-1$ such that $e_{i}, -e_{i}\in B$. Indeed differently the cardinality of $B$ would be at most $l-1$, a contradiction. Finally choosing $1\leq i\leq l-1$ such that $e_{i}, -e_{i}\in B$ we take that $0=e_{i}+(-e_{i})\notin A$, a contradiction again.
\end{proof}
As Theorem \ref{th1.3} comes as a consequence of Theorem \ref{th1.2}, essentially with the same proof the next theorem is a consequence of Theorem \ref{th1.4}.
\begin{thm}\label{th3.3}
Let $X$ be a finite dimensional Banach space with $\dim X=n$. Then for any Auerbach basis $\set{(x_{k},x_{k}^{*}): 1\leq k \leq n}$ of $X$ there exist $n$-linear combinations $z_{1}, \ldots, z_{n}$ of the vectors $x_{1}, \ldots, x_{n}$ with coordinates $0,\pm 1$ such that $\norm{z_{k}}=1$ for $1\leq k \leq n$ and $\norm{z_{k}+z_{l}}>1$ for $1\leq k< l\leq n$.
\end{thm}
We now investigate the complex version of Theorem \ref{th1.2}. For this purpose we introduce (again) the proper notation.\\ \\
\begin{notat}\label{not2.4}
Let $N, l\in \mathbb{N}$. We will write $N\xrightarrow[\Complex]{} l$, if there is no subset $A$ of $V_{N}=\set{0,\pm 1, \pm \mathrm{i}}^{N}$, where $\mathrm{i}=\sqrt{-1}$ such that:
\begin{itemize}
\item [$(a)$] $e_{k}\in A$, for $1\leq k\leq N$
\item [$(b)$] if $x\in A$, then $\mathrm{i}x\in A$
\item [$(c)$] for every subset $B$ of $A$ with $|B|=l$ there exist $x\neq y\in B$ such that $x-y\in A$.
\end{itemize}
$K_{\Complex}(l)$ denotes the minimal $N\in\mathbb{N}$ such that $N\xrightarrow[\Complex]{}l$. We notice that $K_{\Complex}(1)=K_{\Complex}(2)=K_{\Complex}(3)=K_{\Complex}(4)=1$ and that $K_{\Complex}(l)\leq K_{\Complex}(l+1)$, for $l\in\mathbb{N}$.
\end{notat}
\begin{prop}\label{prop3.5}
The function $K_{\Complex}(l)$ is well defined. Moreover $K_{\Complex}(l)\leq\frac{l-1}{2}$ for every odd number $l\geq3$.
\end{prop}
\begin{proof}
Let $n\in\mathbb{N}$ and $A$ be a subset of $V_{n}$ satisfying conditions $(a)$ and $(b)$ of \ref{not2.4}. We consider the $\Real$-linear isomorphism $f:\Complex^{n}\to \Real^{2n}$ with
\[
f(a_{1}+b_{1}\mathrm{i}, \ldots, a_{n}+b_{n}\mathrm{i})=(a_{1},b_{1}, \ldots, a_{n},b_{n}).
\]
It is obvious that $f(V_{n})\subseteq C_{2n}$. Since $A$ satisfies conditions $(a)$ and $(b)$ of \ref{not2.4}, it is easy to check that $f(A)\subseteq C_{2n}$ and that $f(A)$ satisfies conditions $(a)$ and $(b)$ of \ref{not1.3}. So by Theorem \ref{th2.8} there exists $B\subseteq f(A)$ such that $|B|=2n+1$ and $x-y\notin f(A)$ for every $x\neq y\in B$. Now it is direct that $x-y\notin A$ for every $x\neq y \in f^{-1}(B)$. Thus we have that $n\xrightarrow[\Complex]{}2n+1$ and consequently that $K_{\Complex}(n)\leq K_{\Complex}(2n+1)\leq n$, for $n\in\mathbb{N}$.
\end{proof}
\begin{prop}\label{prop3.6}
$K_{\Complex}(2n+1)=n$, for $n\geq1$.
\end{prop}
\begin{proof}
Let $l=2n+1$. If $n=1$, we observe that $K_{\Complex}(3)=1$. Let $n>1$ and assume that $K_{\Complex}(l)<\frac{l-1}{2}$. Then
\begin{equation}\label{eq3.1}
K_{\Complex}(l)\leq\frac{l-1}{2}-1=\frac{l-3}{2}
\end{equation}
Let now $A$ be a subset of $C_{\frac{l-3}{2}}$ satisfying conditions $(a)$ and $(b)$ of 2.3 such that $0\notin A$. We put $\Delta=A\cup\mathrm{i}A$. The set $\Delta$ has the following properties:
\begin{equation}\label{eq3.2}
e_{k}\in\Delta, \,for\,\,\,1\leq k\leq\frac{l-3}{2}
\end{equation}
\begin{equation}\label{eq3.3}
if\,\,\,x, y\in\Delta, then \,\,\,\mathrm{i}x\in\Delta
\end{equation}
\begin{equation}\label{eq3.4}
if \,\,x, y\in\Delta \,\,such\,\, that\,\, x-y \in \Delta, \,\,then\,\, either\,\, x,y\in A\,\, or\,\, x,y \in \mathrm{i}A.
\end{equation}
Properties \eqref{eq3.2} and \eqref{eq3.3} are immediate from the definition of $\Delta$. For property \eqref{eq3.4} let $x,y\in\Delta$ such that $x-y\in\Delta$. Since $\Delta=A\cup\mathrm{i}A$ and $A$ is a subset of $C_{n}$ the coordinates of $x-y$ can be either real or imaginary. Therefore either the coordinates of $x$ and $y$ are real or the coordinates of $x$ and $y$ are imaginary. The last yields the conclusion.

Let $B$ be a subset of $\Delta$, with $|B|=l$ such that $x-y\notin\Delta$, for every $x\neq y\in B$. Then $|B|=|B\cap A|+|B\cap \mathrm{i}A|$. Since $A$ is a subset of $C_{\frac{l-3}{2}}$ satisfying conditions $(a)$ and $(b)$ of \ref{not1.3}, we have that $|B\cap A|$, $|B\cap \mathrm{i}A|\leq\frac{l-3}{2}+1$. So $|B|\leq2(\frac{l-3}{2}+1)=l-1$, a contradiction.
\end{proof}
\begin{prop}\label{prop3.7}
$K_{\Complex}(2n+2)\leq n$, for every $n\in\mathbb{N}$.
\end{prop}
\begin{proof}
As in the preceding proof we consider a subset $A$ of $C_{n}$ satisfying conditions $(a)$ and $(b)$ of \ref{not1.3} such that $0\notin A$. Observe that $A\cap \mathrm{i}A=\emptyset$ and consider the set $\Delta=A\cup\mathrm{i}A$. Theorem \ref{th2.8} yields that there exists a subset $B$ of $A$, with $|B|=n+1$, such that $x-y\notin A$, for every $x\neq y \in B$. Putting $\Gamma=B\cup\mathrm{i}B$, observe that $x-y\notin \Delta$ for every $x\neq y\in \Gamma$. Since $|\Gamma|=2n+2$, we take that $K_{\Complex}(2n+2)\leq n$.
\end{proof}
\begin{thm}\label{th3.8}
$K_{\Complex}(2n+v)=n$, for every $n\in\mathbb{N}$ and $v=1,2$.
\end{thm}
\begin{proof}
The conclusion follows easily from Propositions \ref{prop3.6}, \ref{prop3.7} and the inequality $K_{\Complex}(2n+1)\leq K_{\Complex}(2n+2)$.
\end{proof}
As an immediate consequence of Theorem \ref{th3.8} we get the following.
\begin{thm}\label{th3.9}
For every $n\in\mathbb{N}$ and each subset $A$ of $V_{n}$ satisfying conditions
\begin{itemize}
\item [$(a)$] $e_{k}\in A$ for $1\leq k\leq n$ and
\item [$(b)$] if $x\in A$, then $\mathrm{i}x\in A$,
\end{itemize}
there exists a subset $B$ of $A$ with $|B|=2n+2$ which is difference-free with respect to $A$.
\end{thm}
As in the cases of Theorem \ref{th1.3} and Theorem \ref{th3.3} the above results imply the following Theorem.
\begin{thm}\label{th3.10}
Let $X$ be a finite dimensional complex Banach space with $\dim X=n$. Then for any Auerbach basis $\set{(x_{k},x_{k}^{*}): 1\leq k\leq n}$ of $X$ there exist $(2n+2)$-linear combinations $z_{1}, \ldots, z_{2n+2}$ of the vectors $x_{1}, \ldots, x_{n}$ with coordinates $0,\pm 1,\pm \mathrm{i}$ such that $\norm{z_{k}}=1$, for $1\leq k\leq 2n+2$ and $\norm{z_{k}-z_{l}}>1$ for $1\leq k<l\leq 2n+2$.
\end{thm}
\underline{\textit{Remarks}.} (1) Theorem \ref{th3.10} applied to complex Banach spaces of dimension $n$, gives more information that Theorem \ref{th1.3} for real Banach spaces of dimension $2n$. Indeed, let $X$ be a complex Banach space with $\dim X=n$. Then by Theorem \ref{th3.10} we can find at least $2n+2$ norm-one vectors which are 1-separated. On the other hand, if we apply Theorem \ref{th1.3} on the underlying real Banach space $X$ (which has dimension $2n$ over $\Real$) we can find at least $2n+1$ norm-one vectors which are 1-separated. \\
(2) The uncountable analogue of Kottman's Lemma (and of its consequence which is Theorem \ref{th2.2}) is false. Indeed, it can be shown using $\Delta$-system Lemma as in Remark (2) of \cite{Elton} that, if $\Gamma$ is any infinite set and $\set{x_{\alpha}: \alpha \in A}$ is any set of norm-1 finitely supported vectors of $c_{0}(\Gamma)$ with $\norm{x_{\alpha}-x_{\beta}}>1$, for $\alpha\neq\beta \in A$, then $A$ must be countable. On the other hand, if $\Gamma$ is uncountable it can be shown by transfinite induction, the existence of an $\omega_{1}$-sequence $\set{x_{\alpha}:\alpha<\omega_1}$ in the unit sphere of $c_{0}(\Gamma)$ so that $\norm{x_{\alpha}-x_{\beta}}>1$, for $\alpha<\beta<\omega_1$.

\vspace{0.5cm}
UNIVERSITY OF ATHENS, DEPARTMENT OF MATHEMATICS, \\PANEPISTIMIOUPOLIS, 15784 ATHENS, GREECE\\
\textit{E-mail address: e.glakousakis@gmail.com}\\ \\ \\
UNIVERSITY OF ATHENS, DEPARTMENT OF MATHEMATICS, \\PANEPISTIMIOUPOLIS, 15784 ATHENS, GREECE\\
\textit{E-mail address: smercour@math.uoa.gr}

\begin{thebibliography}{00}
\bibitem{Elton} J. Elton and E. Odell, The unit sphere of every infinite dimensional normed linear space contains a $(1+\varepsilon)$-separated sequence, \textit{Colloq. Math.}, \textbf{44} (1981), 105-109.
\bibitem{Hajek} Hajek P., Montesinos Santalucia V., Vanderwerff J. and Zizler V., Biorthogonal systems in Banach spaces, \textit{CMS Books in Mathematics/ Ouvrages de Mathematique de la SMC} \textbf{26}, Springer, New York, 2008.
\bibitem{Kottman} Kottman C.A., \textit{Subsets of the unit ball that are separated by more than one}, Studia Mathematica \textbf{53} (1975), 15-27.
\bibitem{Arias} Arias-de-Reyna L., Ball K. and Villa R., \textit{Concentration of the distance in finite dimensional normed spaces}, Mathematika \textbf{45} (1998), 245-252.
\end{thebibliography}
\end{document}